\documentclass{amsart}

\usepackage[utf8]{inputenc}
\usepackage[T1]{fontenc}
\usepackage{amssymb,amsmath,amsthm,amsfonts,enumerate}
\usepackage{array}
\usepackage{tikz}
\usepackage{graphicx}
\usepackage{xcolor}
\usepackage[colorlinks, citecolor=black, linkcolor=black, pdfstartview=FitB]{hyperref}
\usepackage{float}          % Pour figures où on veut (mettre [h] )
\usepackage{subcaption}     % Pour def environnement subfigure

\newtheorem{theorem}{Theorem}[section]
\newtheorem{lemma}[theorem]{Lemma}
\newtheorem{proposition}[theorem]{Proposition}

\newtheorem{corollary}[theorem]{Corollary}
\theoremstyle{remark}
\newtheorem{remark}[theorem]{Remark}
\newtheorem*{remark*}{Remark}
\theoremstyle{definition}

\let\Re\undefined\DeclareMathOperator{\Re}{Re}
\DeclareMathOperator{\Log}{Log}
\newcommand{\ds}{\displaystyle}

\title{The Aluthge and the mean transforms of $m$-isometries}
\author{Fadil Chabbabi}
\author{Maëva Ostermann}
\address{Department of Mathematics, Faculty of Sciences Tetouan, Abdelmalek Essaadi University, B. P. 2121 Tetouan, Morocco}
\email{f.chabbabi@uae.ac.ma}
\address{D\'epartement de math\'ematiques et de statistique, Universit\'e Laval,
Qu\'ebec City (Qu\'ebec),  Canada G1V 0A6.}
\email{maeva.ostermann@mat.ulaval.ca}
\date{}

\begin{document}
\maketitle
\begin{abstract}
Let $T\in B(H)$ be a bounded linear operator on a Hilbert space $H$,
let $T = V|T|$ be its polar decomposition of $T$ and let $\lambda\in [0,1]$. The $\lambda$-Aluthge transform $\Delta_{\lambda}(T)$ and  the mean transforms $M(T)$ are  defined  respectively by: 
\[\Delta_{\lambda}(T):=|T|^{\lambda}V|T|^{1-\lambda} \;\; \text{and} \;\; M(T):=\frac12(|T|V+V|T|).\] 
In this paper, we use several examples of weighted shift operators to prove that the Aluthge and mean transforms do not preserve the class of  $m-$isometries in any directions. %Moreover, we show that there exists a $2-$isometry $T$ such that $\Delta_\lambda(T)$ is not a $m$-isometry for any $m\geq 2$ and $\lambda\in(0,1)$. We have the same results for the mean transform, there  exists an  operator $T$ not a $m$-isometry but its mean transform is an isometry.

\end{abstract}
\section{Introduction}
Let $H$ be a complex Hilbert space and $B(H)$ be the algebra of the linear and bounded operator on $H$. For $T\in B(H)$, we consider its polar decomposition $T=V|T|$, where $V$ is the associate partial isometric verifying $Ker(V)=Ker(T)$ and $|T|=(T^*T)^{1/2}$ is the module of $T$. Then this partial isometric $V$ verifies
\[V^*V=P_{\overline{Im(|T|)}}\quad \text{and}\quad VV^*=P_{\overline{Im(T)}}~,\]
where $P_M$ denotes the orthogonal projection on the closed subspace $M$. It is well known that if $T=V|T|$ is the polar decomposition of $T$ then, $|T^*|=V|T|V^*$ and $T^*=V^*|T^*|$ is the polar decomposition of $T^*$.

Related to the polar decomposition of an operator, we consider the Aluthge transform, introduced by A. Aluthge in \cite{Aluthge1990} to study some spectral properties of $p$-hyponormal operators. For an operator $T\in B(H)$ and $T=V|T|$ be its polar decomposition, the Aluthge transform of  $T$ is defined by $\Delta(T)=|T|^{1/2}V|T|^{1/2}$. This definition has been  generalized in \cite{Aluthge1996}. For $\lambda\in[0,1]$, the $\lambda$-Aluthge transform of $T$ is  defined  by \[\Delta_\lambda(T):=|T|^\lambda V|T|^{1-\lambda}.\] Note that $\Delta_0(T)=T$, $\Delta_{1/2}(T)=\Delta(T)$ is the classical Aluthge transform and $\Delta_1(T)=\tilde T$ is the Duggal transform. The Aluthge transform has been studied by many authors, see \cite{AntezanaMasseyStojanoff2005,ChabbabiMbekhta2017,ChabbabiMbekhta2020,Garcia2008,JungKoPearcy2000,JungKoPearcy2001,Wang2003}.

The mean transform of an operator $T=V|T|$ was introduced in \cite{LeeLeeYoon2014} and defined by 
\[M(T):=\frac12(|T|V+V|T|)=\frac12(T+\tilde T),\]
i.e. as the arithmetic mean of $T$ and the Duggal transform $\tilde T$ of $T$.
For more details on the mean transform, for example see \cite{BenhedaChoKoLee2020,BenhidaCurtoLee2019,ChabbabiCurtoMbekhta2019,ChabbabiMbekhta2017,ChabbabiMbekhta2019,ChabbabiMbekhta2020,Lee2016,LeeLeeYoon2014,Yamazaki2019preprint,Zamani2021}. 
For a $\lambda\in[0,1]$, we consider the generalized $\lambda$-mean transform defined by
\[M_\lambda(T)=\frac12\Big(\Delta_\lambda(T)+\Delta_{1-\lambda}(T)\Big),~T\in B(H).\]
Note that $M_\lambda=M_{1-\lambda}$ and that for all $T\in B(H)$, $M_0(T)=M(T),~M_{1/2}(T)=\Delta(T)$.

This paper is inspired by the paper of Botelho and Jamison \cite{BotelhoJamison2010Al}. More precisely, their counterexamples concern the fact that none of the implications between ``$T$ is a 2-isometry'' and ``The Aluthge transform of $T$ is a 2-isometry'' hold. Concretely, we generalize this result for m-isometries and the $\lambda$-Aluthge transform in a fist time and for the mean transform in a second time.

Let $m\ge1$. We say that an operator $T\in B(H)$ is a $m$-isometry if 
\[I+\sum_{k=1}^m(-1)^k\binom{m}{k}(T^*)^kT^k=0,\]
that is equivalent to say that, for all $x\in H$, 
\[\sum_{k=0}^m(-1)^k\binom{m}{k}\|T^kx\|^2=0.\]
Note that if $m=1$, the 1-isometries are exactly the isometries. The class of $m$-isometries has also been well studied, we refer to \cite{AglerStankus1995I,AglerStankus1995II,AglerStankus1996} for more detailed information on the m-isometries.

After recalling some results on weighted shifts in section \ref{Section:WeightedShifts}, we use some examples of these operators to show a several results related to the Aluthge, mean transforms and the class of $m-$isometries. More precisely, in section \ref{Section:Aluthge}, we show  in Theorem \ref{T:AntecedantAluthge} that there exists an operator $T$ which is not a $m$-isometry for any $m\ge1$, such that its Aluthge transform $\Delta(T)$ is an isometry. In Theorem \ref{T:Aluthge}, we prove that there exists  a $2$-isometry $T$ such that its $\lambda$-Aluthge transform $\Delta_\lambda(T)$ is not a $m$-isometry for any $m\ge 2$ and $\lambda\in (0,1)$. In section \ref{Section:Mean},  we show similar results for the mean transform and the class of 2-isometries. In section \ref{section:T&MeanOrAluthge}, we prove that a weighted shift and its Aluthge or mean tranform cannot be simultaneously $2$-isometries, except if the weighted shift is an isometry.
\section{Preliminary results on the weighted shifts}\label{Section:WeightedShifts}
In this section, we recall some results on weighted shifts that will be useful in the next sections. We also take the opportunity to recall some proofs of these results.
Let $l^2$ be the Hilbert space of complex value sequences $x=(x_j)_{j\ge1}$ such that $\|x\|^2=\sum_{j=1}^\infty |x_j|^2$ is finite. Let $(e_j)$ be the canonical basis of $l^2$, i.e. $e_j$ is the sequence defined by \[e_j(n)=\begin{cases}
1&\text{if }j=n,\\
0&\text{if }j\neq n.
\end{cases}
\]
The unilateral shift $S$ is the bounded linear operator on $l^2$ satisfying $Se_j=e_{j+1}$. For a bounded sequence $(a_j)_{j\ge1}$, we define the weighted shift associated to $(a_j)_{j\ge1}$ as the operator $T\in B(l^2)$ satisfying $Te_j=a_je_{j+1}$. 
\subsection{The image of a weighted shift by the mean and the Aluthge transforms}

\begin{lemma}
Let $T\in B(l^2)$ is the weighted shift associated to a bounded sequence $(a_j)_j$. Then the polar decomposition of $T$ is $T=V|T|$ where 
\begin{itemize}
    \item $V$ is the weighted shift associated to the sequence $(\zeta_j)$ defined by $\zeta_j=0$ if $a_j=0$ and $\zeta_j=e^{i\theta_j}=a_j/|a_j|$ if $a_j\neq0$.
    \item $|T|$ is the diagonal operator defined by $|T|e_j=|a_j|e_j$.
\end{itemize}
\end{lemma}

\begin{proof}
Let $A$ be the diagonal operator defined on $l^2$ by $Ae_j=|a_j|e_j$ and $V$ be the weighted shift associated to the sequence $(\zeta_j)$. Clearly $V$ is a partial isometry such that $Ker(V)=Ker(T)$ and $T=VA$. We just have to verify that $A=|T|$. So let $j\ge 1$, since we have that $T^*e_{j+1}=\overline{a_j}e_j$, we deduce that $T^*Te_j=T^*(a_je_{j+1})=|a_j|^2e_j=A^2e_j$. 
\end{proof}

\begin{proposition}
Let $T$ be the weighted shift associated to $(a_j)$. Let $\zeta_j=0$ if $a_j=0$ and $\zeta_j=a_j/|a_j|$ if $a_j\neq0$ Then $M(T)$ is the weighted shift associated to $(m_j)$ with
\[m_j=\zeta_j\frac{|a_j|+|a_{j+1}|}{2}.\]
\end{proposition}

This is a particular case of Proposition 2.2 in \cite{LeeLeeYoon2014}. 

\begin{proof}
Let $T=V|T|$ be the polar decomposition of $T$. 
\[M(T)e_j=\frac12(Te_j+|T|Ve_j)=\frac12(a_je_{j+1}+|T|\zeta_je_{j+1})=\frac{\zeta_j}2(|a_j|+|a_{j+1}|)e_{j+1}.\]
\end{proof}

\begin{proposition}
Let $T$ be the weighted shift associated to a sequence $(a_j)$ with $a_j\neq0$ and $\theta_j\in\mathbb R$ such that $a_j=e^{i\theta_j}|a_j|$. Let $\lambda\in[0,1]$. Then the $\lambda$-Aluthge transform of $T$ is the weighted shift associated to $(b_n)$
with $b_j=e^{i\theta_j}|a_j|^{1-\lambda}|a_{j+1}|^{\lambda}$. 
\end{proposition}

See \cite{Aluthge1990,BotelhoJamison2010Al} for the case $\lambda=1/2$ and \cite{Aluthge1996} for the general case.

\begin{proof}
Let $T=V|T|$ be the polar decomposition of $T$. Then, for $j\ge1$, we have
\begin{align*}
\Delta_\lambda(T)e_j=|T|^\lambda V|T|^{1-\lambda}e_j
&=|T|^\lambda V\big(|a_j|^{1-\lambda}e_j\big)\\
&=|T|^{\lambda}\big(e^{i\theta_j}|a_j|^{1-\lambda}e_{j+1}\big)\\
&=e^{i\theta_j}|a_j|^{1-\lambda}|a_{j+1}|^{\lambda}e_{j+1}=b_je_{j+1}.
\end{align*}
\end{proof}

As a direct consequence of the image of the weighted shifts by the generalized $\lambda$-Aluthge transform, we have the following result.

\begin{corollary}Let $\lambda\in [0,1]$ and let $T$ be the weighted shift on $l^2$ associated to a sequence $(a_j)$, with $a_j>0$ for all $j\ge1$. Then $M_\lambda(T)$ is the weighted shift associated to $(m_{\lambda,j})$ with 
\[m_{\lambda,j}=\frac12\Big(a_j^\lambda a_{j+1}^{1-\lambda}+a_j^{1-\lambda}a_{j+1}^\lambda\Big).\]
\end{corollary}

\subsection{Characterization of weighted shifts that are m-isometries}
The following result was remarked in \cite{BotelhoJamison2010mIso}.

\begin{theorem}
Let $T$ be the weighted shift associated to a sequence $(a_j)$ and $m\ge1$. Then $T$ is a $m$-isometry if and only if 
\[1+\sum_{k=1}^m(-1)^k\binom{m}{k}\prod_{j=0}^{k-1}|a_{j+n}|^2=0,~n\ge1.\]
\end{theorem}

\begin{proof}
Let $n\ge1$ then
 \[T^k(e_n)=\left(\prod_{j=0}^{k-1}a_{j+n}\right)e_{n+k}.\]
We deduce that  $(T^k(e_n))_n$ is an orthogonal family and then for $x=\sum x_ne_n\in l^2$, we have that
\[\|T^kx\|^2=\sum_{n=0}^\infty|x_n|^2\|T^ke_n\|^2.\]
Then we deduce that
\begin{align*}
T \text{ is a $m$-isometry}
&\iff~\forall x\in l^2,~ \sum_{k=0}^m(-1)^k\binom{m}{k}\|T^kx\|^2=0\\
&\iff~\forall n\ge1,~\sum_{k=0}^m(-1)^k\binom{m}{k}\|T^ke_n\|^2\\
&~\hspace{3.9cm}=1+\sum_{k=1}^m(-1)^k\binom{m}{k}\prod_{j=0}^{k-1}|a_{j+n}|^2=0.
\end{align*}
\end{proof}
When $m=2$, we deduce the following characterization of the weighted shift that are $2$-isometries.
\begin{theorem}\label{T:Shift2isoCarac}
Let $T$ be the weighted shift associated to a sequence $(a_j)_{j\ge1}$. Suppose that $T$ is not an isometry. Then $T$ is a $2$-isometry if and only if there exists $a>-1$ such that 
\[|a_j|=\sqrt{\frac{j+1+a}{j+a}},~j\ge1.\]
\end{theorem}
\begin{proof} Without loss of generality, suppose that $a_j\ge0$. In this case, the condition ``$T$ is not an isometry'' become $T\neq S$ which mean that $a_j\neq1$ for some $j$.\\
Suppose first that there exists $a>-1$ such that 
\[a_j=\sqrt{\frac{j+1+a}{j+a}},~j\ge1.\]
Then for all $j\ge1$, we have
\begin{multline*}
    a_{j+1}^2a_j^2-2a_j^2+1=\frac{j+2+a}{j+a}-2\frac{j+1+a}{j+a}+1\\=\frac{j+2+a-2(j+1+a)+j+a}{j+a}=0.
\end{multline*}
Then we deduce that $T$ is a $2$-isometry.

Suppose now that $T$ is a $2$-isometry. Suppose there exists $j\ge1$ such that $a_j=1$. Since
\begin{equation}\label{Eq:2iso}
a_{j+1}^2a_j^2-2a_j^2+1=0,~j\ge1,
\end{equation}
we can deduce that $a_j=1$ for any $j\ge1$ and thus $T=S$. Then, suppose that $T\neq S$. Thus $a_j\neq1$ for all $j\ge1$. Moreover, since the equation \ref{Eq:2iso} is equivalent to $a_j^2(2-a_{j+1}^2)=1$, we have that $a_j\neq0$ and then if we divide this equality by $a_j^2$ and put $u_j=(a_j^2-1)^{-1}$, we obtain the following relation:
\[u_{j+1}=\frac1{a_{j+1}^2-1}=\frac{1}{\ds~1-\frac1{a_j^2}~}=\frac{a_j^2}{a_j^2-1}=u_j+1.\]
Therefore $(u_j)$ is an arithmetic sequence and thus there exists $a$ such that
\[u_j=j+a=\frac1{a_j^2-1},\text{~then~}a_j^2=\frac{j+1+a}{j+a},\]
and thus 
\[a_j=\sqrt{\frac{j+1+a}{j+a}},~j\ge1.\]
Finally, since $\frac{j+1+a}{j+a}>0$ for any $j\ge1$, we have that $a>-1$.
\end{proof}
\section{The Aluthge transform and m-isometries}\label{Section:Aluthge}

In this section, we generalized results of  Bethelo and Jamison in \cite{BotelhoJamison2010Al} for m-isometries and the generalized Aluthge transform. The following theorem gives the existence of operators that are not $m$-isometries such that their Aluthge transforms are isometries.

\begin{theorem}\label{T:AntecedantAluthge}
There exists an operator $T\in B(l^2)$ that is not a $m$-isometry, for any $m\ge2$ such that $\Delta(T)$ is an isometry.
\end{theorem}

The case $m=2$ is done by Bethelo -- Jamison \cite{BotelhoJamison2010Al}. We prove that them example works for any $m\ge2$.

\begin{proof}
Let $T$ be the weighted shift on $l^2$ associated to the sequence defined by $a_{2j+1}=1/2$ and $a_{2j}=2$ for  $j\ge1$. Then
\[\prod_{j=0}^{k-1}|a_{j+1}|^2=\prod_{j=1}^k|a_j|^2=
\begin{cases}
1&\text{ if $k$ is even,}\\
\ds\frac12&\text{ if $k$ is odd.}
\end{cases}\]
Let $m\ge1$, then
\begin{align*}
1+\sum_{k=1}^m(-1)^k\binom{m}{k}\prod_{j=0}^{k-1}|a_{j+1}|^2
   &=1+\sum_{\underset{k~\text{even}}{1\le k\le m}}\binom{m}{k}-\frac12\sum_{\underset{k~\text{odd}}{1\le k\le m}}\binom{m}{k}\\
   &=\sum_{\underset{k~\text{even}}{0\le k\le m}}\binom{m}{k}-\frac12\sum_{\underset{k~\text{odd}}{1\le k\le m}}\binom{m}{k}\\
   &=2^{m-1}-2^{m-2}=2^{m-2}>0.
\end{align*}
Then $T$ is not a $m$-isometry. But its Aluthge transform is $\Delta(T)=S$
that is an isometry on $l^2$.
\end{proof}

The following theorem gives the existence of operators that are not $2$-isometries such that there
$\lambda-$Aluthge transforms are isometries.

\begin{theorem}\label{T:AntecedantLambdaAluthge}
Let $\lambda \in(0,1)$. There exists an operator $T_\lambda\in B(l^2)$ that is not a $2$-isometry such that $\Delta_\lambda(T_\lambda)$ is an isometry.
\end{theorem}

\begin{proof}
Suppose that $a_j\ge 0$ and $a_j^{1-\lambda} a_{j+1}^{\lambda}=1$. Then
\[a_j=x^{\left(\frac{\lambda-1}{\lambda}\right)^j}~\text{for some}~x>0.\]
Let $\lambda \in[1/2,1)$, then $\left|\frac{\lambda-1}{\lambda}\right|\le 1$ and thus for any choice of $x>0$, the weighted shift $B_x$ associated to $\left(x^{\left(\frac{\lambda-1}{\lambda}\right)^j}\right)_j$ is a bounded operator on $l^2$ and $\Delta_\lambda(B_x)=S$. To finish the proof, we verify that $B_x$ is not a 2-isometry for some choice of $x$.
To show that, we could notice that the weight of $B_x$ cannot be put in the form given in Theorem \ref{T:Shift2isoCarac}. We will rather here make a more direct calculation. Suppose that $B_x$ is a 2-isometry then
\begin{align*}
0&=1-2x^{2\left(\frac{\lambda-1}{\lambda}\right)}+x^{2\left(\frac{\lambda-1}{\lambda}\right)}x^{2\left(\frac{\lambda-1}{\lambda}\right)^2}\\
&=1-2x^{2\left(\frac{\lambda-1}{\lambda}\right)}+x^{2\frac{\lambda(\lambda-1)+(\lambda-1)^2}{\lambda^2}},
\intertext{and this is true if and only if}
2x^{2\frac{\lambda-1}{\lambda}}&=1+x^{2\frac{\lambda(\lambda-1)+(\lambda-1)^2}{\lambda^2}}.
\end{align*}
Let $f(x)=2x^{2\frac{\lambda-1}{\lambda}}$ and $g(x)=1+x^{2\frac{\lambda(\lambda-1)+(\lambda-1)^2}{\lambda^2}}$ for $x>0$. Then
\begin{align*}
    f'(1)&=4\frac{\lambda-1}{\lambda},\intertext{and}
    g'(1)&=2\frac{\lambda(\lambda-1)+(\lambda-1)^2}{\lambda^2}.
\end{align*}
Since $\lambda\neq1$ then $f'(1)\neq g'(1)$ and in particular $f\not\equiv g$ on $(0,1)$. 

If $\lambda\in(0,1/2)$ then $B_x$ is bounded for $0<x<1$, and we can remark that 
\[\lim_{x\to0}f(x)=+\infty\neq \lim_{x\to0}g(x)=1.\]

Then, in the two cases,  there exists $x_\lambda\in (0,1)$ such that $f(x_\lambda)\neq g(x_\lambda)$ and thus such that $T_\lambda:=B_{x_\lambda}$ is not a 2-isometry. 
\end{proof}
The following theorem gives the existence of $2$-isometries such that their $\lambda$-Aluthge transforms are not $m$-isometries.

\begin{theorem}\label{T:Aluthge}
There exists a 2-isometry $T\in B(l^2)$ such that $\Delta_\lambda(T)$ is not a $m$-isometry, for any $m\ge2$ and any $\lambda\in(0,1)$.
\end{theorem}

The case $m=2$ and $\lambda=1/2$ is done by Bethelo -- Jamison \cite{BotelhoJamison2010Al}. We prove that them example works for any $m\ge2$ and any $\lambda\in(0,1)$.

\begin{proof}
Let 
\[a_j=\sqrt{\frac{j+1}{j}},~j\ge1.\]
Then the weighted shift $T$ associated to $(a_n)$ is a $2$-isometry on $l^2$ by Theorem \ref{T:Shift2isoCarac}. Fix $\lambda\in(0,1)$ and let \[b_j=\left(\frac{j+1}j\right)^{(1-\lambda)/2}\left(\frac{j+2}{j+1}\right)^{\lambda/2}.\] Then the $\lambda$-Aluthge transform of $T$ is the weighted shift associated to $(b_n)$. Let $m\ge2$ and $n\ge1$. Then
\begin{align*}
1+\sum_{k=1}^m(-1)^k\binom{m}{k}\prod_{j=0}^{k-1}|b_{j+n}|^2
   &=1+\sum_{k=1}^m(-1)^k\binom{m}{k}\left(\frac{k+n}n\right)^{1-\lambda}\left(\frac{k+n+1}{n+1}\right)^{\lambda}\\
   &=1+\sum_{k=1}^m(-1)^k\binom{m}{k}\left(1+k/n\right)^{1-\lambda}\left(\frac{1+(k+1)/n}{1+1/n}\right)^{\lambda}.
\end{align*}
Let $\Omega_k=\left\{z~:~\Re(z)>\frac{-1}{k+1}\right\}$ and, for $1\le k \le m$, define
\[u_{k,1}(z)=1+kz,~u_{k,2}(z)=1+(k+1)z~\text{and}~u_{k,3}(z)=1+z\quad(z\in\Omega_k).\]
Then $u_{k,1},u_{k,2},u_{k,3}\in Hol(\Omega_k)$ and $\Re(u_{k,l}(z))>0$ for all $z\in\Omega_m$. Let $\Log$ the principal determination of the logarithm and let $v_k\in Hol(\Omega_k)$ be defined by \[v_k=u_{k,1}^{1-\lambda} u_{k,2}^{\lambda}u_{k,3}^{-\lambda}=e^{(1-\lambda)\Log(u_{k,1})+\lambda\Log(u_{k,2})-\lambda\Log(u_{k,3})}.\]
Then $v_k$ verify
\[v_k(x)=\left(1+kx\right)^{1-\lambda}\left(\frac{1+(k+1)x}{1+x}\right)^{\lambda}\quad(x\ge0).\]
And define $f\in Hol(\Omega_m)$ by 
\[f=1+\sum_{k=1}^m(-1)^k\binom{m}{k}v_k.\]
Then, for all $n\ge1$,
\begin{align*}f\left(\frac1n\right)&=1+
\sum_{k=1}^m(-1)^k\binom{m}{k}\left(1+k/n\right)^{1-\lambda}\left(\frac{1+(k+1)/n}{1+1/n}\right)^{\lambda}\\&=1+\sum_{k=1}^m(-1)^k\binom{m}{k}\prod_{j=0}^{k-1}|b_{j+n}|^2.\end{align*}
Suppose that $\Delta_\lambda(T)$ is a $m$-isometry, then $f(1/n)=0$ for all $n\ge1$. Since $0\in\Omega_m$, then $f\equiv0$. In this case, 
\[v_m=(-1)^m\left(-1-\sum_{k=1}^{m-1}(-1)^k\binom{m}{k}v_k\right)\in Hol(\Omega_{m-1})\]
and thus if we define \[g(z)=e^{\lambda\Log(1+(m+1)z)}=v_3(z)(1+z)^{\lambda}(1+mz)^{\lambda-1}\]
then $g\in Hol(\Omega_{m-1})$. \\
Let $\varepsilon>0$ such that $D:=\overline{D}(-1/(m+1),\varepsilon)\subset\Omega_{m-1}$. Then
\begin{align*}
    \frac1{2i\pi}\int_{\partial D}g(z)dz
&=\frac1{2i\pi}\int_{\partial D}(1+(m+1)z)^{\lambda}dz\\
&=\frac{\varepsilon^{1+\lambda}}{2\pi}\int_{-\pi}^\pi e^{i\theta(1+\lambda)}d\theta\\
&=\frac{\varepsilon^{1+\lambda}}{2i\pi(1+\lambda)}(e^{i\pi(1+\lambda)}-e^{-i\pi(1+\lambda)})\\
&=\frac{-\varepsilon^{1+\lambda}}{\pi(1+\lambda)}\sin(\lambda\pi).
\end{align*}
And thus \[\frac1{2i\pi}\int_{\partial D}g(z)dz\neq0~~\text{if}~~\lambda\in(0,1).\] This give a contradiction with $g\in Hol(\Omega_{m-1})$ by the Cauchy-Goursat Theorem. And then, $\Delta_\lambda(T)$ is not a $m$-isometry.
\end{proof}

\section{The mean transform and m-isometries}\label{Section:Mean}
In this section, we prove similar results of the previous section but for the mean transform. The following theorem gives the existence of operators that are not $m$-isometries such that their mean transforms are isometries.

\begin{theorem}\label{T:AntecedentIsometrieMeanTr}
There exists $T\in B(l^2)$ which is not a $m$-isometry for any $m$ such that $M(T)=S$.
\end{theorem}

\begin{proof}
Let $T$ be the weighted shift associated to $(a_j)$ with 
\[a_j=\begin{cases}
\frac12&\text{if $j$ is even,}\\
\frac32&\text{if $j$ is odd.}\\
\end{cases}\]
Then $M(T)$ is the weighted shift associated to $(m_j)$ with
\[m_j=\frac{a_j+a_{j+1}}{2}=\frac{\frac12+\frac32}{2}=1.\]
We deduce that $M(T)= S$. Now, we can verify that $T$ is not a $m$-isometry for any $m$. To do that, we can remark that $T^2=\frac34 S^2$. Then 
\[\rho(T)^2=\rho(T^2)=\frac34\implies \rho(T)=\frac{\sqrt3}2.\]
But the spectrum of a $m$-isometry is either included in the unit circle or is the entire closed unit disk  (See Lemma 1.21 in \cite{AglerStankus1995I}). In particular, its spectral radius is $1$ so $T$ can not be an $m$-isometry.
\end{proof}

The counterexample in the proof allows us to determine other properties which are not preserved by the mean transform, in particular the spectral radius is not preserved (proved in Example 4.1 in \cite{ChabbabiCurtoMbekhta2019}) and this remark allow us to deduce the following result.

\begin{corollary}
There exists an operator $T\in B(l^2)$ which is similar to a contraction but such that $M(T)$ is not a power bounded operator. 
\end{corollary}

\begin{remark} By the von Neumann inequality, an operator similar to a contraction is completely polynomially bounded (See \cite{vonNeumann1951} and \cite{Paulsen2002} for details).
Then a direct consequence of this corollary is that the mean transform do not preserve power bounded, polynomially bounded and completely polynomially bounded operators.
\end{remark}

\begin{proof}
Let $T$ like in the proof of Theorem \ref{T:AntecedentIsometrieMeanTr} and $\alpha\in(1,2/\sqrt3)$.
Then \[\rho(\alpha T)=\alpha\sqrt3/2<1\] and by the Rota's theorem (See \cite{Rota1960} or Corollary 9.14 in \cite{Paulsen2002}), $\alpha T$ is similar to a contraction.
But \[\rho(M(\alpha T))=\rho(\alpha M(T))=\alpha\rho(M(T))=\alpha>1. \] 
Then \[\|M(T)^n\|\to\infty~~\text{when }n\to\infty,\]
and thus $M(T)$ is not power bounded.
\end{proof}
Let us consider $T$ the counter-example given in the proof of Theorem \ref{T:Aluthge}, i.e. the weighted shift associated to the sequence defined by
\[a_j=\sqrt{\frac{j+1}j},~j\ge1.\]
We can remark that the mean transform of $T$ will not be a $2$-isometry. Indeed, since $M(T)$is the wighted shift associated to the sequence $m_j=(a_j+a_{j+1})/2$, we have that 
\begin{align*}
\|M(T)^2e_1\|^2-2\|M(T)e_1\|^2+1
&=m_1^2m_2^2-2m_1^2+1\\
&=\left(\frac{2+\sqrt3}{2\sqrt2}\right)^2\left(\frac{3+2\sqrt2}{2\sqrt6}\right)^2-2\left(\frac{2+\sqrt3}{2\sqrt2}\right)^2+1\\
&=\frac{48\sqrt6+68\sqrt3-108\sqrt2-337}{184}\\
&<\frac{48\times3+68\times2-108-337}{184}=-\frac{165}{184}<0.
\end{align*}
This operator will also give a counter-example for the case of the generalized mean transform, i.e. we have the following result.

\begin{theorem}\label{T:Mgene}
Let $\lambda\in[0,1]$. There exists a 2-isometry $T\in B(H)$ such that $M_\lambda(T)$ is not a 2-isometry.
\end{theorem}

\begin{proof}
Let $T$ be the weighted shift associated to the sequence $(a_j)$ with 
\[a_j=\sqrt{\frac{j+1}{j}},~j\ge1.\]
Then $M_\lambda(T)$ is the weighted shift associated to $(m_{\lambda,j})$ with 
\begin{align*}
m_{\lambda,j}
&=\frac12\Big(a_j^\lambda a_{j+1}^{1-\lambda}+a_j^{1-\lambda}a_{j+1}^\lambda\Big)\\
&=\frac12\left(\left(\frac{j+1}j\right)^{\lambda/2} \left(\frac{j+2}{j+1}\right)^{(1-\lambda)/2}+\left(\frac{j+1}j\right)^{(1-\lambda)/2}\left(\frac{j+2}{j+1}\right)^{\lambda/2}\right).
\end{align*}
It is not easy to verify that $M_\lambda(T)$ cannot be written as in Theorem \ref{T:Shift2isoCarac}, then we do a direct proof.
We know that $M_\lambda(T)$ is a 2-isometry if and only if for all $j\ge1$, $m_{\lambda,j+1}^2m_{\lambda,j}^2-2m_{\lambda,j}^2+1=0$ and we have
\begin{align*}
4m_{\lambda,j}^2
&=\left(\left(\frac{j+1}j\right)^{\lambda/2} \left(\frac{j+2}{j+1}\right)^{(1-\lambda)/2}+\left(\frac{j+1}j\right)^{(1-\lambda)/2}\left(\frac{j+2}{j+1}\right)^{\lambda/2}\right)^2\\
&=\left(\frac{j+1}j\right)^{\lambda} \left(\frac{j+2}{j+1}\right)^{1-\lambda}+\left(\frac{j+1}j\right)^{1-\lambda}\left(\frac{j+2}{j+1}\right)^{\lambda}\\
    &~\hspace{6cm}+2\left(\frac{j+1}j\right)^{1/2} \left(\frac{j+2}{j+1}\right)^{1/2}\\
&=\left(1+\frac1j\right)^{2\lambda-1}\left(1+2\frac1j\right)^{1-\lambda}+\left(1+\frac1j\right)^{1-2\lambda}\left(1+2\frac1j\right)^{\lambda}\\
    &~\hspace{6cm}+2\left(1+\frac1j\right)\left(1+2\frac1j\right)^{1/2}.
\end{align*}
With the help of the principal holomorphic determination of the logatithm on $\mathbb C\backslash\mathbb R^-$, we define $g$, the holomorphic function on $\Omega_1:=\mathbb C\backslash(-\infty,-1/2]$, by 
\[g(z):=\left(1+z\right)^{2\lambda-1}\left(1+2z\right)^{1-\lambda}+\left(1+z\right)^{1-2\lambda}\left(1+2z\right)^{\lambda}+2\left(1+z\right)\left(1+2z\right)^{1/2}.\]
We have that $g(1/j)=4m_{\lambda,j}^2$ and $g\left(\frac{1/j}{1+1/j}\right)=4m_{\lambda,j+1}^2$. Then,
let $\Omega_2:=\{z\in\mathbb C:z/(1+z)\in\Omega_1\}$, $\Omega=\Omega_1\cap\Omega_2$ and let $h$ be the holomorphic function on $\Omega$ defined by
\[h(z)=g(z)g\left(\frac{z}{1+z}\right)-8g(z)+16,~z\in\Omega.\]
Suppose that $M_\lambda(T)$ is a 2-isometry. Then for all $n\ge1$, we have that $h(1/n)=0$. Since $[0,+\infty)\subset \Omega$, by the isolated zero principle, $h=0$ on $\Omega$ and then on $[0,+\infty)$.
Moreover, we have that $\lim\limits_{x\to+\infty}g(x)=+\infty$ then if the equality 
\[g(x)\left(8-g\left(\frac{x}{1+x}\right)\right)=16\]
is true for all $x\ge0$, we must have that \[g(1)=\lim\limits_{x\to+\infty}g\left(\frac{x}{1+x}\right)=8.\] So let
\[u(s)=2^{2s-1}3^{1-s}+2^{1-2s}3^{s}+4\sqrt3,~s\in[0,1].\]
Then 
\[u'(s)=2^{-1 - 2 s} 3^{-s)} (-4\times 9^s + 3\times16^s) \Log(4/3),\]
which is negative on $(0,1/2)$ and positive on $(1/2,1)$.
Then 
\[8<6\sqrt3=u(1/2)\le u(\lambda)=g(1)\le u(0)=\frac72+4\sqrt3.\]
In particular, $g(1)\neq8$ and it gives a contradiction so $M_\lambda(T)$ is not a 2-isometry.\\~
\end{proof}
\section{Characterization of $2$-isometric weighted shifts with a 2-isometric Aluthge or mean transform}\label{section:T&MeanOrAluthge}
In this section, we consider $2$-isometric weighted shifts such that their Aluthge or mean transform are $2$-isometries. We remark that such operators are necessarily isometries.

\begin{theorem}\label{T:T&Delta(T)2iso}
Let $T$ be the weighted shift associated to a sequence $(a_n)$. If $T$ and $\Delta(T)$ are 2-isometries then $T$ is an isometry and in particular $T=\Delta(T)=M(T)$.
\end{theorem}

\begin{proof}
Without loss of generality, suppose that $a_j>0$. Suppose on the contrary that $T$ is not an isometry i.e. $a_j\neq1$. Then there exists $a>-1$ such that 
\[a_j=\sqrt{\frac{j+1+a}{j+a}},~j\ge1.\]
If $\Delta(T)$ is an isometry, then $\Delta(T)=S$. In this case, we should have that $a_ja_{j+1}=1$ but the form of $a_j$ implies that $a_j>1$ for all $j\ge1$. Then there exists $b>-1$ such that $\Delta(T)$ is the weighted shift associated to the sequence defined by
\[\sqrt{a_ja_{j+1}}=\sqrt{\frac{j+1+b}{j+b}},~j\ge1.\]
This relation implies that, for all $j\ge1$,
\[\left(\frac{j+1+b}{j+b}\right)^2=\frac{j+2+a}{j+a}.\]
This equality is true if and only if for all $j\ge1$, 
\[(j+1+b)^2(j+a)=(j+b)^2(j+2+a),\]
which is equivalent to
\begin{multline*}
j^3+(a+2(b+1))j^2+(b+1)(b+1+2a)j+a(b+1)^2\\=j^3+(a+2+2b)j^2+(2a+4+b)bj+(a+2)b^2.
\end{multline*}
To have this equality on these polynomials, we need to have the equality of the coefficients and then 
\[\begin{cases}
a(b+1)^2=(a+2)b^2,
\\(b+1)(b+1+2a)=(2a+4+b)b,
\end{cases}\iff
\begin{cases}
2b^2=2ab+a,\\2b+2a+1=4b.
\end{cases}\]
With the second condition, we deduce that $b=a+1/2$. Then the first condition becomes $2a^2+2a+1/2=2a^2+a+a$. Finally, since $1/2\neq0$, these conditions are not verified and then $T$ has to be an isometry.
\end{proof}

The following theorem gives the same result in the case of the mean transform.

\begin{theorem}\label{T:T&M(T)2iso}
Let $T$ be the weighted shift associated to a sequence $(a_j)$. If $T$ and $M(T)$ are 2-isometries then $T$ is an isometry and in particular $T=\Delta(T)=M(T)$.
\end{theorem}
\begin{proof}Without loss of generality, suppose that $a_j>0$.
Suppose on the contrary that $T$ is not an isometry i.e. $a_j\neq1$. Then there exists $a>-1$ such that 
\[a_j=\sqrt{\frac{j+1+a}{j+a}},~j\ge1.\]
Since $a_j>1$, then $a_j+a_{j+1}>2$ and thus $M(T)$ is not an isometry. Then there exists $b>-1$ such that $M(T)$ is the weighted shift associated to the sequence 
\[\frac{a_j+a_{j+1}}2=\sqrt{\frac{j+1+b}{j+b}},~j\ge1.\]
From this equality, we deduce that 
\[\frac{j+1+a}{j+a}+\frac{j+2+a}{j+a+1}+2\sqrt{\frac{j+2+a}{j+a}}=4\frac{j+1+b}{j+b},\]
and then
\[\left(\frac{j+1+a}{j+a}+\frac{j+2+a}{j+a+1}-4\frac{j+1+b}{j+b}\right)^2=4\frac{j+2+a}{j+a}.\]
Then we have
\begin{multline*}
\Big((j+b)(j+1+a)^2+(j+a)(j+b)(j+2+a)-4(j+a)(j+a+1)(j+b+1)\Big)^2
\\=4(j+a+2)(j+a)(j+a+1)^2(j+b)^2.
\end{multline*}
This condition gives an equality between two polynomials and in particular, we can replace $j$ by $z\in\mathbb C$. If we choose $z=-a$, then this equality is that $a=b$. Then this equality gives
\[(z+a)^2\Big(2z^2+4(a+1)z+2a^2+4a+3\Big)^2=4(z+a)^3(z+a+1)^2(z+a+2).\]
If this equality is true then the zeroes of $2z^2+4(a+1)z+2a^2+4a+3$ are real. But the discriminant of this polynomial is $16(a+1)^2-4\times 2\times(2a^2+4a+3)=-8<0$. Then this inequality gives a contradiction. Thus finally, $T$ has to be an isometry.
\end{proof}
\begin{remark}
Theorems \ref{T:T&Delta(T)2iso} and \ref{T:T&M(T)2iso} implies that a weighted shift and its Aluthge or mean tranform cannot be simultaneously $2$-isometries, except if the weighted shift is an isometry. Do we have the same result for general operators, for $m$-isometries and generalized Aluthge and mean transforms? 
\end{remark}
\bibliographystyle{plain}
\bibliography{biblio}

\end{document}